\documentclass[reqno,11pt]{amsart}
\usepackage{latexsym, amsfonts,mathrsfs, amsmath, amssymb, amscd, epsfig}

\usepackage{color}
\def\clrr{\color{red}}
\newtheorem{thm}{Theorem}
\newtheorem{lem}[thm]{Lemma}
\newtheorem{corollary}[thm]{Corollary}

\theoremstyle{definition}

\newtheorem{prop}[thm]{Proposition}
\newtheorem{remark}{Remark}

\newenvironment{pf}{{\noindent \it \bf Proof. }}{{\hfill$\Box$}\\}

\def\hinv{{1\over h}}

\def\cQ{{\mathcal Q}}

\def\cKg{{\mathcal K} }

\def\gm{  \overline{\gamma}}
\def\gp{ (\gm+{\textstyle{1\over2}})}
\def\Obw{ {\Omega}^{\text{bw}}_{N,T}}
\def\Ofw{{\Omega}^{\text{fw}}_{N,T}}

\def\ep{\varepsilon}
\def\ba{\begin{array}}
\def\ea{\end{array}}
\def\bma{\left(\begin{matrix}}
\def\ema{\end{matrix}\right)}
\def\be{\begin{equation}}
\def\ee {\end{equation}}
\def\bse{\begin{subequations}}
\def\ese{\end{subequations}}

\def\wt{\widetilde}

\def\bbR{{\mathbb R}}
\def\CR{{M_0}}
\def\CRe{\CR_\ep}

\def\pt{\partial_t}

\def\px{\partial_x}
\def\pxi{\partial_{\xi}}

\def\pxixi{\partial_{\xi\xi}}

\def\mR{{\mathbb R}}

\def\cF{{\mathcal F}}

\def\somef{f}
\def\Tm{T}

\def\PVin{\omega_0}
\def\PVins{\omega_0^\sharp}

\def\Wins{W_0^\sharp}

			\def\Zins{Z_0^\sharp}
\def\Zf{Z^\flat}

\def\Gin{G_0}	\def\Ein{E_0}
\def\als{\alpha^\sharp}

\def\hs{{h^*_0}}

\def\htt{h^{\gamma+1\over2}}

\def\IC{\Big|_{t=0}}

\begin{document}

\title[ ]{Singularity formation and global existence of Classical Solutions for One Dimensional Rotating Shallow Water System}

\author[ ]{Bin Cheng}
\address{Department of Mathematics, University of Surrey, Guildford, United Kingdom}
\email{b.cheng@surrey.ac.uk}
\author[]{Peng Qu}
\address{School of Mathematical Sciences, Fudan University, Shanghai, 200433, China}
\email{pqu@fudan.edu.cn}
\author[ ]{Chunjing Xie}
\address{Department of mathematics, Institute of Natural Sciences, Ministry of Education Key Laboratory of Scientific and Engineering Computing, SHL-MAC, Shanghai Jiao Tong University, 800 Dongchuan Road, Shanghai, 200240, China}
\email{cjxie@sjtu.edu.cn}

\keywords{Rotating shallow water system, formation of singularity, Riemann invariants, global existence, Klein-Gordon equation}
\subjclass[2010]{35L40, 35L67, 35L72, 35Q86, 76U05}


\begin{abstract}
We study classical solutions of one dimensional rotating shallow water system which plays an important role in  geophysical fluid dynamics. The main results contain  two contrasting aspects. First,
when the solution crosses certain threshold, we prove finite-time singularity formation for the classical solutions by studying the weighted gradients of Riemann invariants and utilizing conservation of physical energy. In fact,  the singularity formation will take place for a large class of ${C}^1$ initial data  whose gradients and physical energy can   be   arbitrarily small and  higher order derivatives should be large. Second, when the initial data have constant potential vorticity,  global existence of small classical   solutions  is established  via studying  an equivalent form of  a quasilinear Klein-Gordon equation satisfying certain null conditions. In this global existence result, the smallness condition is in terms of the higher order Sobolev norms of the initial data.
\end{abstract}

\maketitle

\section{Introduction and Main Results}

 The one-dimensional rotating shallow water system plays an important role  in the study of geostrophic adjustment and zonal jets (e.g. \cite{Zeitlin, Galperin}).
 Upon suitable rescaling, the one-dimensional rotating shallow water system in  Eulerian form reads
 \begin{equation} \label{RE:Eu}
 \left\{
\begin{aligned}
&\pt h+\px ( h u)=0,\\
&\pt( h u)+ \px( h u^2)+  \px { h}^\gamma/\gamma= h v, \\
&\pt( h v)+ \px( h u v)=- h u,
\end{aligned}
\right.
\end{equation}
where $ h$  denotes the height of the fluid surface,   $u$
denotes the velocity component in the $x$-direction, and $v$ is the other horizontal velocity component that is in the direction orthogonal to the $x$-direction. The Coriolis force, caused by a rotating frame, is represented by the $( hv,-hu)^T$ terms on the right-hand side of the last two equations of \eqref{RE:Eu}. For the rotating shallow water model, one has $\gamma = 2$  in the pressure law (\cite{CF}), but in this paper we will prove results that are valid for the general case   $\gamma\geq 1$.

The system \eqref{RE:Eu} is a typical one dimensional system of balance laws, which has  attracted plenty of studies   since Riemann \cite{Dafermos}. One of the most important features of { nonlinear} hyperbolic system of conservation laws is that the wave speed depends on the solution itself so that the classical solutions in general are expected to form singularity  in finite time (cf. \cite{Lax, John}).
In fact, the  system \eqref{RE:Eu} can also be regarded as one dimensional compressible Euler system with source terms. The formation of singularity and critical threshold phenomena for the compressible Euler system without or with certain special source terms were studied in \cite{LiuTad, TW, CPZ} and references therein. The additional source terms very often demand substantial novel techniques in addition to the classical singularity formation theory developed in   \cite{Lax, John} and subsequent literature.

On the other hand, with $(h, u, v)$ depending only on $(t,x)$-variables,  the  system \eqref{RE:Eu} can be regarded as a special case of the two dimensional rotating shallow water system (\cite{Pedlosky}). The latter is a widely used  approximation of the three dimensional incompressible Euler equations and the Boussinesq equations in the regime of large scale geophysical fluid motion. In the two dimensional setting, it is shown in \cite{ChengXie}  that there exist global classical solutions for a large class of small initial data subject to the constant potential vorticity constraint, which is analogue of the irrotational constraint for the compressible Euler equations. Since   classical solutions of two dimensional compressible Euler system in general form singularity (c.f. \cite{Rammaha}), the   result of \cite{ChengXie}
evidences that the Coriolis forcing term plays a decisive role in the global well-posednesss theory for   classical solutions of   compressible flows.
A key ingredient in the proof of \cite{ChengXie}    is that the   rotating shallow water system and in fact also its one dimensional reduction,
 when subject to the (invariant) constant potential vorticity constraint, can be reformulated into a quasilinear Klein-Gordon system.
Note that the solutions of Klein-Gordon equations are of faster dispersive decay than those of the corresponding nonlinear wave system, the latter of which is derived from non-rotating
 fluid models. The rate of this dispersive decay, however, is tied to spatial dimension. In fact, in the one dimensional setting,
 the decay rate of Klein-Gordon system is not fast enough for  the global existence
 theory of \cite{ChengXie} to be directly applied to the system \eqref{RE:Eu}.

In short, for   studying the lifespan and global existence of classical solution for the one dimensional rotating shallow water system, we have to introduce novel techniques and investigate the nonlinear structure of the system \eqref{RE:Eu} more carefully.

The main objectives of this paper are two-fold. First, we study the formation of singularities for a general class of $C^1$ initial data by capturing  the nonlinear interactions in the system  \eqref{RE:Eu}. These initial data can have arbitrary small gradients and physical energy but higher order derivatives should be large. Second, we take a careful look at the structure of the system \eqref{RE:Eu}, exploit the dispersion provided by the Coriolis forcing terms, and show the global existence of classical solutions for a class of small initial data that are of small size in terms of higher order Sobolev norms.

Before stating the main theorems, we first rewrite the system \eqref{RE:Eu} in the Lagrangian coordinates, which takes a  simpler form in one dimensional setting.

Assume the initial height field $h(0,x)={h}_0(x)\in C^1(\mR)$ is strictly away from vacuum, i.e., $0<c_h\le {h}_0(x)=h(0,x)\le C_h$. It then induces   a  ``coordinate stretching'' at $t=0$ specified by a $C^2$ bijection $\phi:\mR \to\mR$ defined as
\be\label{relabel}  \xi=\phi(x):= \int_0^{ x} h(0, s)\,d s,\ee
whose inverse function   can be written as $x=\phi^{-1}(\xi)$.
Assume $ u\in C^1([0,T)\times\mR)$. Let $\sigma (t,\xi)$ be the unique particle path determined by
\begin{equation}\label{ODE:xit}
\left\{
\begin{aligned}
&\pt \sigma (t, \xi) = u(t, \sigma (t, \xi)),\\
&\sigma (0,\xi)=\phi^{-1}(\xi),
\end{aligned}
\right.
\end{equation}
so that for each fixed $t\in[0,T)$, we have   a $C^1$ bijection   $x=\sigma(t,\xi):\mR \to \mR$.
Define
\begin{equation}\label{def:huv}
\wt {h} (t, \xi) := h(t, \sigma (t, \xi)),\quad \wt {u}(t, \xi): = u(t, \sigma (t, \xi)), \quad \text{and} \quad \wt{v}(t, \xi) := v(t, \sigma (t, \xi)).
\end{equation}

In the Appendix \ref{secaa}, we show that $(h, u, v)$ solves the system \eqref{RE:Eu} if and only if $(\wt h, \wt u, \wt v)$ defined in \eqref{def:huv} is a solution of the following
rotating shallow water system in the Lagrangian form,
 \begin{equation} \label{RSW:Lag}
 \left\{
\begin{aligned}
&\pt  \wt h+ \wt h^2 \pxi \wt u=0,\\
&\pt  \wt u + \pxi \wt h^\gamma/\gamma - \wt v=0, \\
&\pt \wt v +  \wt u=0.
\end{aligned}
\right.
\end{equation}

For the rest of the paper, we deal with the  system \eqref{RSW:Lag}. For convenience, we drop the tilde signs  in $\wt h$, $\wt u$, $\wt v$ when there is no ambiguity for the presentation. Thus,
in the Lagrangian coordinates, the rotating shallow water system \eqref{RSW:Lag} can be written as
\be \label{p:hinv}
\left\{
\begin{aligned}
&\partial_t h +h^2\pxi  u=0, \\
&\partial_t u+\pxi  \big(h^\gamma/ \gamma\big)-v=0,  \\
&\partial_t v +u=0\,.
\end{aligned}
\right.
\ee
%
The   objective in   this paper is then to study the Cauchy problem for the system \eqref{p:hinv} with initial data
\begin{equation}\label{RSW:IC}
(h,u, v)(0,\xi) = (h_0(\xi), u_0(\xi), v_0(\xi))\qquad\text{for }\;\;\xi\in\mR \,.
\end{equation}
 For any $C^1$ solution of \eqref{p:hinv}, it follows from \eqref{p:hinv} that one has
\be \label{PV:Lag}
\pt\Big(\hinv  +\pxi v\Big)=0.
\ee
This is a key geophysical property of rotating fluid which is  conservation of potential vorticity.
Therefore, we have the invariance of the  potential vorticity
\be\label{ICvorticity}
{1\over h(t,\xi)}+\pxi v(t, \xi)= {1\over h(0,\xi)}+\pxi v(0, \xi):=\PVin(\xi)\,
\ee
in the  Lagrangian form.

Before stating the main theorems, some notations are in order. Inspired by the techniques of \cite{TW} by Tadmor and Wei, we introduce the ``weighted gradients of Riemann invariants'',
\begin{equation}\label{def:Z}
Z_j=\sqrt{h}\,\big[\pxi u+(-1)^j h^{\frac{\gamma-3}{2}}\pxi h\big]\quad \text{for}\,\, j=1,2.
\end{equation}
Also, define
\begin{equation}\label{defZins}
\Zins:=\sup_\xi \max_{1\leq i \leq 2} Z_i(0,\xi) \quad \text{and}\quad
\PVins:=\sup_{\xi}\PVin(\xi)\,.
\end{equation}
The first main result   is on the formation of singularities for classical solutions and consists of two theorems.
 \begin{thm}\label{thm:large:sing}
Fix $\Tm'\ge0$. Consider a classical solution  $(h,u,v)\in C^1([0,\Tm']\times\bbR)$ to the rotating shallow water    system \eqref{p:hinv} with initial data satisfying
 $\inf_{\xi }h_0>0$ and $(h_0-1, u_0, v_0)\in C_0^1(\mR)$.
 If
 \be\label{Zj:large:neg} \inf_\xi\min_{1\leq i \leq 2} Z_i(T',\xi)\le -\sqrt{2\PVins}    \,,\ee
then the solution  must develop a singularity in finite time $t=T^\sharp>T'$ in the following sense
\be\label{def:sing}
\inf_{\substack{0\leq t< T^\sharp,\\ \xi\in \mR}} h(t, \xi)>0,\quad \sup_{\substack{0\leq t< T^\sharp,\\ \xi\in \mR}}\max \{h(t, \xi),  |u(t,\xi)|, |v(t, \xi)|\}< \infty,
\ee
and
\be\label{def:sing2}
\sup_{\substack{0\leq t< T^\sharp,\\ \xi\in \mR}}\max_{1\leq j \leq 2} Z_j(t, \xi)<\infty, \quad
  \lim_{t\nearrow T^\sharp}\inf_{\xi\in \mR} \min_{1\leq j\leq 2} Z_j(t, \xi)=-\infty\,.
\ee
\end{thm}

The proof of Theorem \ref{thm:large:sing} is given in Section \ref{sec31}.
\begin{remark}
It follows from the conservation of potential vorticity and the bounds for $h$ in \eqref{def:sing} that $\pxi v$ is bounded even when the singularity is formed. Furthermore, it follows from the definition of $Z_j$ ($j=1$, $2$) in \eqref{def:Z} and the estimates in \eqref{def:sing}-\eqref{def:sing2} that we have
\[
\lim_{t\nearrow T^\sharp}\inf_{\xi\in \mR} \pxi u(t, \xi)=-\infty\,.
\]
\end{remark}

Obviously, if the initial data satisfy \eqref{Zj:large:neg}, then the solution of the problem \eqref{p:hinv} form singularities in finite time. In fact, we can also characterize  a class of initial data which does not satisfy
 \eqref{Zj:large:neg} at the initial time, but rather evolve to satisfy  \eqref{Zj:large:neg}, and eventually form a singularity according to the above theorem.

We define physical energy of the rotating shallow water system  as
 \be\label{E:cons}E(t):=\int_{-\infty}^\infty{1\over2}(u^2+v^2)(t,\xi)+\cQ(h(t,\xi))\,d\xi,\ee
where
\begin{equation}\label{defcQ}
\cQ(h):={1\over\gamma} \int_1^h  ( s^{\gamma-2}-s^{-2})\,ds \ge 0 \,.
\end{equation}
Finally, define
\be\label{def:Gin}
\Ein=E(0)\quad\text{ and }\quad \Gin:= \sqrt{2\PVins}+\max\Big\{\Zins ,\,\sqrt{2 \PVins}\Big\}\,,
\ee
where $\PVins$ and $\Zins$ are defined in \eqref{defZins}.

\begin{thm}\label{mainthm1}
Consider the Cauchy problem  for the system \eqref{p:hinv}  subject to initial data \eqref{RSW:IC} satisfying $(h_0-1,u_0,v_0)\in C_0^1(\mR)$ and
 $\inf_{\xi }h_0>0$.
Suppose  \be\label{mainthm1:cond}\inf_\xi\min_{1\leq i\leq 2} Z_i(0, \xi)< -\sqrt2\,\sqrt{\PVins-\Big[\cF_\gamma^{-1}\big(\Gin\,\Ein\big)+1\Big]^{-{2\over\gamma}}}\,\ee
where  $\cF_\gamma^{-1}$ is the inverse function of the function $\cF_\gamma(\cdot)$ defined by
\be\label{def:cF} \cF_\gamma(\alpha ):{=}{16\over3\gamma^3}\, {\alpha^3\over  ( \alpha+1)^3}\Big\{   (\alpha+1)^{3-\frac{2}{\gamma}}+(\alpha+1)^{3-\frac{2}{\gamma}-1}\Big\}\,.
 \ee
Then the solution must develop a singularity in finite time $t=T^\sharp$ in the sense of \eqref{def:sing} and \eqref{def:sing2}.
\end{thm}

The proof of Theorem \ref{mainthm1} is a straightforward combination of Theorems \ref{thm:large:sing} and \ref{thmenergy}. We have the following remarks.

\begin{remark}
Straightforward computations show that $\cF_\gamma$ defined in \eqref{def:cF} is a monotonically increasing function mapping $(0,\infty)$ to $(0, \infty)$. Hence the inverse function $\cF_\gamma^{-1}$ is always well-defined on $(0,\infty)$.
 \end{remark}
\begin{remark}
If $(h_0-1,v_0)$ are compactly supported, it follows from \eqref{ICvorticity} that one always has $\PVins\ge1$, so all the square roots in  \eqref{Zj:large:neg}, \eqref{def:Gin}, and \eqref{mainthm1:cond} are always real.
\end{remark}

\begin{remark}
We consider only the initial data  such that $(h_0-1,u_0,v_0)$ are compactly supported. As the propagation of information for general data is  at a finite speed, the results in Theorems \ref{thm:large:sing} and \ref{mainthm1} can be easily extended to general initial data without compact support. Also thanks to the finite speed of propagation, when the initial data are indeed compactly supported and a singularity does develop in finite time as in \eqref{Zj:large:neg} and \eqref{def:sing2}, we actually have the singularity occur at a finite location as well.
 \end{remark}

\begin{remark}\label{OE3} {The singularity formation criterion \eqref{mainthm1:cond} allows  arbitrarily small initial gradients at the order of $O(\Ein^{1/3})$.} Indeed, by definition of $\cF_\gamma$ in \eqref{def:cF}, we have
 \[\lim_{\alpha\searrow 0}{\cF_\gamma(\alpha)\over\alpha^3}={32\over 3\gamma^3}\,.\]Therefore, with   $G_0>0$ bounded above by a constant, we can find positive constants $C$ and $\overline E_{0}$ so that
 \be\label{CE3} C^{-1} {\Ein}^{1/3} \leq \cF_\gamma^{-1}\big(\Gin\,\Ein\big)\leq C{\Ein}^{1/3} \quad\text{for all }\;\; E_0<\overline E_0.\ee
Then, by choosing arbitrarily small $E_0<\overline E_0$ and choosing $\PVins$ to be arbitrarily close to 1 (with the most convenient choice being $h_0\equiv 1, v_0\equiv 0$), we make the right hand side of (18) at order $O(E_0^{1/3})$. This in turn allows us to choose initial data having small gradients, i.e.,
$$Z_j\Big|_{t=0} =\inf_x\Big\{\sqrt{h}\,\big[\pxi u+(-1)^j h^{\frac{\gamma-3}{2}}\pxi h\big]\Big\}\IC\sim O(E_0^{1/3})$$
which satisfy the condition (18) for the singularity formation.
\end{remark}

 \begin{remark}
{ The singularity formation criterion \eqref{mainthm1:cond} also reflects the fact that we utilize physical energy and its conservation to prove pointwise singularity formation. }
 \end{remark}

 In fact, if the initial data has not only small gradients, but also small higher derivatives in Sobolev spaces, there is a global existence of classical solutions for rotating shallow water system. This is our next main result on the global existence of classical solutions for rotating shallow water system.
\begin{thm}\label{mainthm2} Consider the Cauchy problem  \eqref{p:hinv} and \eqref{RSW:IC} subject to compactly supported  initial data $(h_0-1,u_0,v_0)$ with
 $\inf_{\xi }h_0>0$.
Suppose that $ \PVin\equiv 1$. Then, there exists a small positive number $\delta$ so that if the   Sobolev norm  $\|  u_0\|_{H^{k}(\mR)}+\|v_0 \|_{H^{k+1}(\mR)}<\delta$ for some sufficiently large integer $k$, then    there is a global classical solution for the problem \eqref{p:hinv} and \eqref{RSW:IC} for all time $t\ge0$.
\end{thm}

Theorem \ref{mainthm2} is proved in Section \ref{sec:KL}.
There are a few remarks in order.

\begin{remark}
In Theorem \ref{mainthm2}, we consider only the data close to the constant state $(1,0,0)$. In fact, the results also hold for any data close to $(\bar H_0, 0, 0)$ with a constant $\bar H_0$.
\end{remark}

 \begin{remark}
 {Although the singularity formation result in Theorem \ref{mainthm1} allows {arbitrarily small initial gradients} at the order of $O(\Ein^{1/3})$ for any sufficiently small $\Ein$,
  Theorem \ref{mainthm1} and Theorem \ref{mainthm2} are compatible, or more precisely, they characterize different sets of initial data. To see this, we recall Gagliardo-Nirenberg interpolation inequality} to have
  \[\|\pxi u_0\|_{L^\infty(\mathbb{R})}\le C\|u_0\|_{H^2(\mathbb{R})} \leq C\|u_0\|_{H^{k}(\mathbb{R})}^{2\over k}\,\|u_0\|_{L^2(\mathbb{R})}^{1-{2\over k}}\le C\,\delta^{2\over k}\,\Ein^{{1\over2}-{1\over k}}\,.\]
For any initial data satisfying the assumptions in Theorem \ref{mainthm2} with $k\geq 7$ so that $\PVin\equiv 1=\PVins$ and small $E_0<\delta$, one has
\[\|\pxi u_0\|_{L^\infty(\mathbb{R})}\le C \delta^{2\over k} \Ein^{5/14}.
\]
Applying similar argument to $1- {1\over h_0}=\PVin - {1\over h_0} = \pxi v_0$ shows that
\[
\|h_0-1\|_{L^\infty(\mR)}+ \|\pxi h_0\|_{L^\infty(\mR)}\le C \delta^{2\over k}  E_0^{5/14},
\]
which is  much smaller than $ O(\Ein^{1/3})$. By the lower bound in \eqref{CE3}, it is {impossible} for such initial data to also satisfy the assumption \eqref{mainthm1:cond} of Theorem \ref{mainthm1} as long as we choose $\delta$ in Theorem \ref{mainthm2} to be sufficiently small.
 \end{remark}

\begin{remark}
By the Theorem \ref{mainthm2} above, with sufficiently small initial data,   there is a global solution for the rotating shallow water system, which is fundamentally different from  the non-rotating, compressible Euler system \cite{John}. This shows that the rotation plays an important role in the well-posedness theory of classical solutions to the partial differential equations modeling compressible   flows.
\end{remark}

\begin{remark}
In fact, the results on both singularity formation  and global existence in this paper also work in the similar fashion for the one dimensional Euler-Poisson system with a nonzero background charge for hydrodynamical model in semiconductor devices and plasmas. We would also like to mention the recent work \cite{GHZ} where the global existence of classical solutions for the Euler-Poisson system with small initial data was proved via a different method.
\end{remark}

The rest of the paper is organized as follows. In Section \ref{sec:Rie}, we introduce the Riemann invariants and weighted gradients of Riemann invariants, and give some basic estimates for these quantities.  In Section \ref{sec:shock}, we  prove the finite time formation of singularity via investigating the weighted gradients of Riemann invariants and utilizing conservation of physical energy.
In Section \ref{sec:KL}, we reformulate the Lagrangian rotating shallow water system subject to constant potential vorticity  into a one dimensional Klein-Gordon equation which is then shown to satisfy the null conditions in \cite{Delort:1D}. The results in \cite{Delort:1D} help establish the global existence of small classical solutions.
We also provide two appendices. Appendix \ref{secaa} contains the proof for the equivalence between the Eulerian form of rotating shallow water system \eqref{RE:Eu} and its Lagrangian form \eqref{RSW:Lag}. In Appendix \ref{secab}, we present two elementary lemmas which are used to prove the singularity formation for the rotating shallow water system.

\section{Riemann Invariants and Their Basic Estimates}\label{sec:Rie}
The system \eqref{p:hinv} is a typical $3\times 3$ system of balance laws. One way to diagonalize the system of balance laws is to write the system in terms of the Riemann invariants.
 However, a $3\times 3$ system usually does not have $3$ full Riemann invariant coordinates \cite{Dafermos}. Fortunately, the system \eqref{p:hinv} has $3$ full Riemann invariant coordinates $R_i$ ($i=1, 2, 3$), so that
 the  system \eqref{p:hinv} is recast into a ``diagonalized'' form,
\be \label{R:sys}
\left\{
\begin{aligned}
&\partial_t R_1 -\htt\pxi  R_1 -R_3=0\,,\\
&\partial_t R_2 +\htt\pxi  R_2 -R_3=0\,,\\
&\pt R_3+{R_1+R_2\over 2}=0\,,
\end{aligned}
\right.
\ee
where, by borrowing notations from the so-called $p$-system, we let $p({1\over h})={h^\gamma\over\gamma}$, i.e., $p(s):={s^{-\gamma}\over\gamma}$ and define Riemann invariants as
\be\label{R:def}\left\{ \begin{aligned}
&R_1:=u+\int_1^\hinv \sqrt{-p'(s)} \,ds=u-\cKg(h),\\
&R_2:=u-\int_1^\hinv \sqrt{-p'( s)}\, ds=u+\cKg(h),\\
&R_3:=v,
\end{aligned}\right.\ee
with, apparently,
\be\label{K:gamma:def}
\cKg(h):= \int_1^h s^{\gamma-3\over2}\,ds\,. \ee
Note that   $h$ can be expressed in terms of the Riemann invariants as
 \be\label{h:form} h= \vartheta\Big({R_2-R_1\over2}\Big)\quad\;\;\text{with}\quad\;\;\vartheta(z) =\cKg^{-1}(z) =\begin{cases}\left({\gamma-1\over2}\,z+1\right)^{2\over\gamma-1},&\gamma>1,\\
e^z,&\gamma=1.\end{cases}\ee

Based on the Riemann invariants formulation alone, we have the following estimates related to the $L^\infty$ bounds of the solutions which then lead to an important upper bound for $h$ and consequently the finite speed of propagation.

\begin{lem}\label{lem:up:h}Fix $T>0$.
 Let   $(h,u,v)\in C^1([0,T]\times\bbR)$ with $h>0$ solve  system \eqref{p:hinv} and equivalently \eqref{R:sys}.  Suppose
 \be\label{main:bd:cond}
\Zins,\;\PVins,\;\inf_{\xi }h_0\;\text{ and }\;\sup_{\xi }\{h_0,|u_0|,|v_0|\}\;\;\; \text{are all finite and positive}\ee
and
  $$\CR  :=\sup_{\xi\in\bbR}\{|R_1(0,\xi)|,|R_2(0,\xi)|,|R_3(0,\xi)|\}<\infty.$$
 Then, at any $t\in[0,T]$, we have
 \be\label{max:R}\sup_{\xi\in\bbR}\{|R_1(t,\xi)|,|R_2(t,\xi)|,|R_3(t,\xi)|\}\le \CR e^t
 \ee
 and
 \be\label{max:h}\sup_{\xi\in\bbR} h(t,\xi)\le\theta^\sharp(t):=\vartheta(\CR  e^t) =\begin{cases}\left({\gamma-1\over2}\,\CR  e^t+1\right)^{2\over\gamma-1},&\gamma>1,\\
e^{ (\CR  e^t )},&\gamma=1.\end{cases} \ee

 \end{lem}
\begin{pf} Obviously, the estimate \eqref{max:h} is a consequence of the representation \eqref{h:form} and the estimate \eqref{max:R}. So we need only to prove the estimate \eqref{max:R}. Then, it suffices to show that,  for any $\ep>0$,   $N>0$, we have
\be \label{max:Xb}
\max_{(t,\xi)\in A_{N,T,\ep} }\max_{1\leq i\leq 3}|e^{-t}R_i|(t,\xi)<  \CRe:=\CR+\ep,
\ee
where $A_{N,T,\ep}$ (see Fig. 1) is the trapezoid
\be\label{def:A:trap} A_{N,T,\ep}:=\Big\{(t,\xi)\in [0,T]\times\bbR\,\Big|\,  |\xi|\le N+(T-t)\big[\vartheta(e^T(\CRe  + \ep))\big]^{\gamma+1\over2}\Big\}.\ee

Suppose that the estimate \eqref{max:Xb} is not true. By the compactness of  $A_{N,T,\ep}$, there must exist an earliest time $t'$ so that
\be\label{earliest}
\max_{(t,\xi)\in A_{N,T,\ep}\atop t\in[0,t']}\max_{1\leq i\leq 3}|e^{-t}R_i|(t,\xi)=  \CRe  .\ee
Note the definition of $\CR$ implies that $t'>0$. The speeds  of  the characteristics of \eqref{R:sys} are  $-h^{\gamma+1\over2}$, $0$, and $h^{\gamma+1\over2}$, respectively. It follows from \eqref{h:form} and \eqref{earliest} that one has
\[h^{\gamma+1\over2}< \big[\vartheta(e^T(\CRe  + \ep) )\big]^{\gamma+1\over2}\quad\text{ for }(t,\xi)\in A_{N,T,\ep}\text{ and }t\in[0,t'].\]
 Thus, the above estimate together with  the definition  of $A_{N,T,\ep}$ in \eqref{def:A:trap} guarantees that  all   characteristics of \eqref{R:sys} emitting from $(t',\xi')$ and going  backward in time always stay within $A_{N,T,\ep}$.  Now, introduce
 \[\begin{aligned}m^\sharp(t):=\max_{ (t,\xi)\in A_{N,T,\ep} }\max_{1\leq i\leq 3}R_i(t,\xi)\,\quad \text{and}\quad
 m^\flat(t):=\min_{ (t,\xi)\in A_{N,T,\ep} }\max_{1\leq i\leq 3}R_i(t,\xi)\,.
 \end{aligned}\]
 Then, for any $t\in(0,t']$, upon integrating each equation of  \eqref{R:sys} along the associated characteristic from $0$ to $t$, we have
 \[m^\sharp(t)\le m^\sharp(0)+\int_0^t\max\{m^\sharp(s),\,-m^\flat(s)\}\,ds\]
 and
 \[
 m^\flat(t)\ge m^\flat(0)+\int_0^t\min\{-m^\sharp(s),\,m^\flat(s)\}\,ds\,.
 \]
 Therefore, we have
 \[\max\{m^\sharp(t),\,-m^\flat(t)\}\le \max\{m^\sharp(0),\,-m^\flat(0)\}+\int_0^t\max\{m^\sharp(s),\,-m^\flat(s)\}\,ds\,.\]
Therefore,  $\max\{m^\sharp(t),\,-m^\flat(t)\}$ satisfies a Gronwall's inequality which leads to
 \[\max\{m^\sharp(t),\,-m^\flat(t)\}\le e^t\,\max\{m^\sharp(0),\,-m^\flat(0)\}\le e^t\,M_0\quad\text{ for all }\;\; t\in(0,t']\,.\]
 This contradicts \eqref{earliest}. Hence the lemma is proved.
\end{pf}

The upper bound of $h$ in  \eqref{max:h} allows us to define the following trapezoidal regions, similar to \eqref{def:A:trap}, in the spirit of domain of dependence and domain of  influence.
\begin{align}\label{def:Ombw:trap} \Obw&:=\Big\{(t,\xi)\in [0,T]\times\bbR\,\Big|\,  |\xi|\le N+(T-t)\big[\vartheta(e^T(\CR  + 1))\big]^{\gamma+1\over2}\Big\},\\
\label{def:Omfw:trap} \Ofw&:=\Big\{(t,\xi)\in [0,T]\times\bbR\,\Big|\,  |\xi|\le N+t\big[\vartheta(e^T(\CR  + 1))\big]^{\gamma+1\over2}\Big\}.\end{align}
\begin{center}
\includegraphics[height=6cm, width=7.5cm]{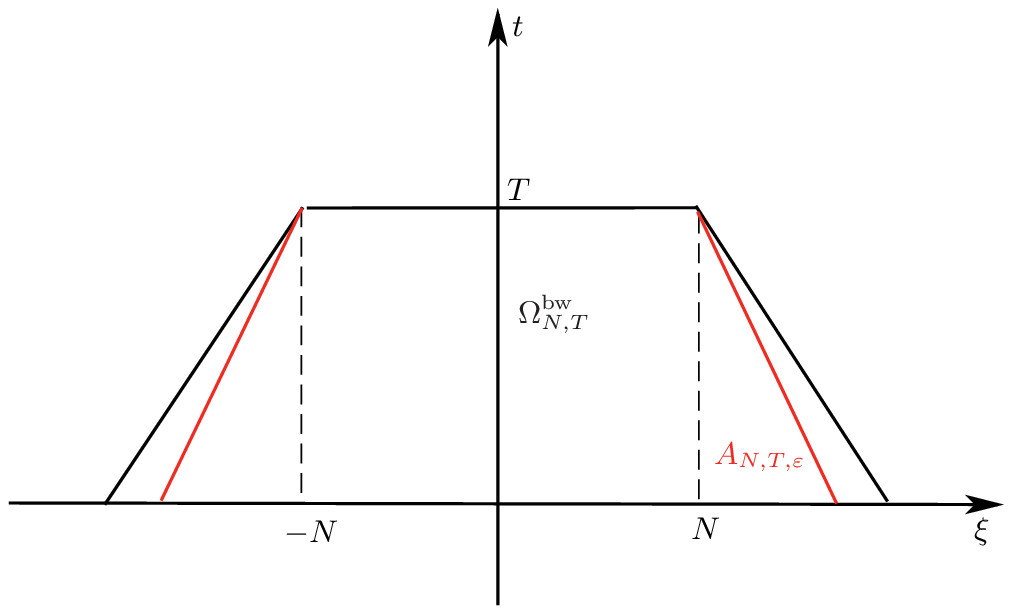}
\includegraphics[height=6cm, width=7.5cm]{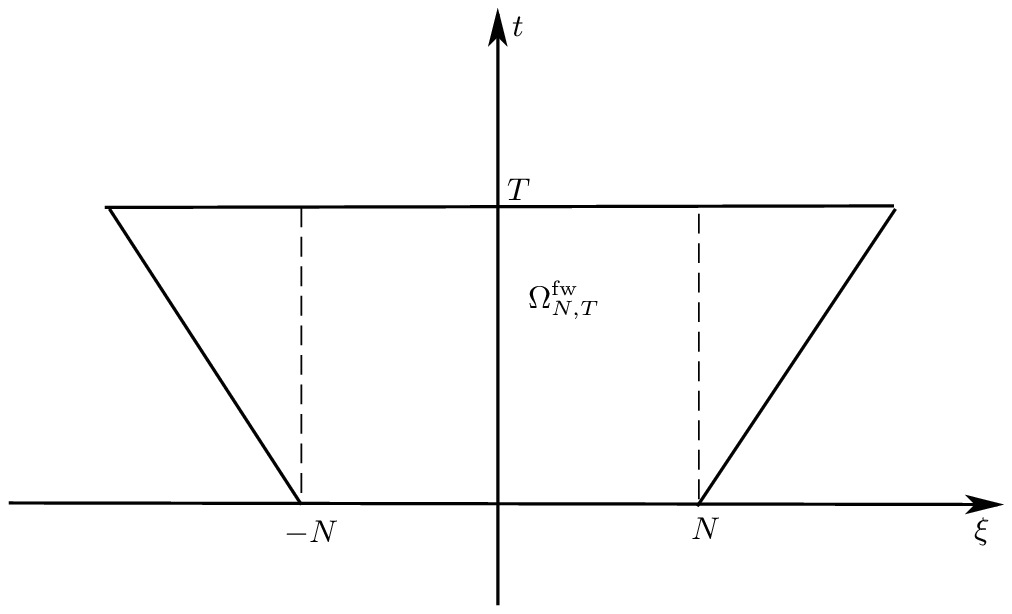}\\
{  Fig. 1\,\,\,\, $\Obw$ and $A_{N,T,\ep}$} \hspace{3cm} {Fig. 2 \,\,\,\, $\Ofw$}

\end{center}

Under the assumptions of Lemma \ref{lem:up:h}, we have that characteristics  with speed $\pm\htt$ or $0$ emitting from within $\Obw$ (resp. $\Ofw$) and going {backward} (resp. {forward}) in time always stay within $\Obw$ (resp. $\Ofw$) till $t=0$ (resp. $t=T$).

Note that it is crucial that we shall also prove the  lower bound of $h$ to be strictly above $0$, which will be dealt with later.
\subsection{Dynamics of gradients of Riemann invariants}
Here, we follow the original idea of Lax (\cite{Lax}) to study the dynamics of gradients of Riemann invariants
and will further reformulate the system inspired by the method in \cite{TW} by Tadmor and Wei. Note that despite the similarity of our equations with those of \cite{TW}, their ODEs \cite[(3.12)]{TW} for weighted gradients of the Riemann invariants do not have a term that corrsponds to $1/h$ term in our ODEs \eqref{Z:n}. This is in fact one of the main technical difficulties we have to tackle here.

 First, differentiating the first
two equations of \eqref{R:sys} with respect to $\xi$ gives 
\be\label{D:R:1}\left\{\begin{aligned}&D_1 (\pxi R_1)-\pxi(\htt)(\pxi R_1)=\pxi R_3,\\
&D_2 (\pxi R_2)+\pxi(\htt)(\px R_2)=\pxi R_3,\end{aligned}\right. \ee
where \[D_1:=\pt-\htt\pxi\quad\text{and}\quad  D_2:=\pt+\htt\pxi.\]

It follows from  \eqref{K:gamma:def} and \eqref{h:form}   that one has
\[
\pxi h= {1\over \cKg'(h)}\,{\pxi R_2 - \pxi R_1 \over2} =h^{-{\gamma-3\over2}}\,{\pxi R_2 - \pxi R_1 \over2},
\]
 so
\[\pxi(\htt)={\gamma+1\over2}h^{\gamma-1\over2}\,\pxi h={\gamma+1\over2}\,h\,{\pxi R_2 - \pxi R_1 \over2}.\]
 Combine this with potential vorticity conservation \eqref{ICvorticity}  and transform \eqref{D:R:1} into
 \begin{equation}\label{Y1Y2}
 \left\{
 \begin{aligned}
 &D_1 (\pxi R_1)={\gamma+1\over4}h\,(\pxi R_2-\pxi R_1)(\pxi R_1)+\PVin-{1\over h},\\
 &D_2 (\pxi R_2)= {\gamma+1\over4}h\,( \pxi R_1-\pxi R_2 )(\pxi R_2)+\PVin-{1\over h}.
 \end{aligned}
 \right.
 \end{equation}

We also use $\pxi R_1,\pxi R_2 $ and the first  equation  in \eqref{p:hinv} to rewrite the dynamics of $h$ as
\be\label{h:RI:grad}\pt h+{h^2(\pxi R_2+\pxi R_1)\over2}=0.\ee
It follows from the definition of $R_1$ and $R_2$ that one has
  \begin{equation}\label{eq441}
  {h^2(\pxi R_2-\pxi R_1)\over2}=h^2 \pxi\cKg(h)=h^{\gamma+1\over2}\pxi h.
  \end{equation}
Substituting \eqref{eq441} into \eqref{h:RI:grad} gives
\be \label{h:Y} D_1 h   =-h^2   \pxi R_2 \,\;\quad\text{ and equivalently }\quad\;
 D_2  h =-h^2 \pxi R_1\,.\ee

Recall the definitions of Riemann invariants in \eqref{R:def} and of weighted gradients of Riemann invariants  $Z_j$ ($j=1$, $2$) in \eqref{def:Z} to rewrite
\[Z_1= \sqrt h\,\pxi R_1 \quad\text{and}\quad  Z_2=\sqrt h\,\pxi R_2\,.\]
Then, combine    \eqref{Y1Y2} and \eqref{h:Y} to derive dynamics of $Z_j$ along the characteristics,
\be \label{Z:n}
D_j Z_j={{\sqrt h}}\left[ -\gp Z_j^2+ \gm Z_1Z_2+ \PVin(\xi) -{1\over h}\right],\qquad j=1,2,\ee
where $\gm={\gamma-1\over4}\ge 0$. Furthermore, it follows from \eqref{h:RI:grad} that we have
\be \label{h:Z}\pt h=-{1\over2}\,h^{3/2}\,(Z_1+Z_2),\;\;\text{ i.e. }\;\;  \pt {1\over {\sqrt h}}   = {1\over4}(Z_1+Z_2).\ee
\subsection{Upper bound for weighted gradients of Riemann invariants}
The following lemma uses the above   formulation in terms of the weighted gradients of Riemann invariants to
show an  upper bound of $Z_j$ and consequently a positive lower bound of $h$.

\begin{lem}\label{lem:up:low}
Fix $T>0$. Under the same assumptions and notations as in Lemma \ref{lem:up:h},
  we have that at any $t\in[0,T]$,
  \be\label{sup:Z}Z_j\le\Wins:=\max\Big\{\Zins ,\,\sqrt{2 \PVins
 }\Big\} \qquad\text{for }\;\; j=1,2, \ee
 and
\be\label{inf:h}  h\ge \Big[{1\over\sqrt{ \inf_{\xi }{  h_0} }}+{t\over2}\Wins \Big]^{-2}   .\ee\end{lem}
\begin{pf} Consider any large but compact region   $\Obw$ as defined in \eqref{def:Ombw:trap}. It is a domain  of dependence for its every time slice. Then it follows from  \eqref{h:Z} and $ \inf_{\xi}h_0>0$ that $h$ is always positive in $\Obw$. Now,
  it suffices to show
 \be\label{sup:Z:1} \displaystyle\max_{\Obw}\{Z_1,Z_2\}\le\Wins\,,\ee
 and
 \be\label{inf:h:1}  \displaystyle\max_{\Obw}{1\over {\sqrt h} }\le {1\over \sqrt{  \inf_{\xi }{  h_0}  }}+{t\over2}\Wins\,.\ee

Noting that $\Zins>0$,
we assume without loss of generality that at some   $(t',\xi')\in\Obw$,
\begin{equation}\label{eq521}
Z_1(t',\xi')=\max_{\Obw}\{Z_1,Z_2\}>0.
\end{equation}
If the maximum in \eqref{eq521} is attained at $t'=0$, then the estimate  \eqref{sup:Z:1} is apparently true. Otherwise, the maximum in \eqref{eq521} is achieved for  $t'>0$. Therefore, one has  $D_1Z_1{(t',\xi')}\ge0$. This, together with \eqref{Z:n}, yields
\[{{\sqrt h}}\left[ -\gp Z_1^2+ \gm Z_1Z_2+ \PVin(\xi) -{1\over h}\right]\ge 0\quad\text{at }(t',\xi').\]
Since $Z_1(t',\xi')\ge Z_2(t',\xi')$ and $Z_1(t', \xi') >0$, we have
\[ -{1\over2} Z_1^2 + \PVins\ge {1\over h}>0\quad\text{at }(t',\xi').\]
This  proves  the estimate \eqref{sup:Z:1}. The estimate \eqref{inf:h:1} is a direct consequence of   equation  \eqref{h:Z}  and the  estimate \eqref{sup:Z:1}.
\end{pf}

 Although the results proved so far are regarding {closed} time interval $[0,T]$, this is not so essential. In fact, for any small $\ep>0$, we can replace every $T$ by $T-\ep$ in Lemmas \ref{lem:up:h} and \ref{lem:up:low} and still obtain the corresponding estimates. Since these estimates are regardless of $\ep$, we can let $\ep$ approach zero and establish that
 all estimates in Lemmas \ref{lem:up:h} and \ref{lem:up:low} are still valid if we replace every occurrence of $[0,T]$ by $ [0,T) $ in their conditions and conclusions.
 Now, we obtain the following corollary that characterizes the type of possible singularities that a solution may develop.

\begin{corollary}\label{cor:Tmax:open}
Given the same type of initial conditions as in Lemma \ref{lem:up:h}, suppose a $C^1$ solution exists over time interval $[0,T]$ (resp. $[0,T)$). Then,  for all $t\in[0,T]$ (resp. $t\in [0,T) $) and all $\xi\in\mR$, we have $(h,|u|,|v|, |\pxi v|)$ { as well as $Z_1$ and $Z_2$} to be   uniformly bounded from above,  and have $h$ to be  uniformly   bounded from below by a positive constant.

Furthermore, if a classical solution indeed loses $C^1$ regularity at a finite time $t=T^\sharp$, then $\inf_{\xi}\{Z_1, Z_2\}\to -\infty$ as $t\to T^\sharp$ while $Z_1$ and $Z_2$ remain bounded from above.
\end{corollary}

\section{Formation of Singularities}\label{sec:shock}

Let us recollect the bounds obtained so far, independent of the size of initial data, we have obtained
 upper bound \eqref{max:h} and positive lower bound \eqref{inf:h} for $h$; and
upper bound \eqref{sup:Z} for $Z_1$ and $Z_2$. Since the solution itself is always bounded as proved in Lemma \ref{lem:up:h}, the only possible singularity for a classical solution is for $Z_1$ or $Z_2$ approaching $-\infty$.

In this section, we first prove a threshold using comparison principle, so that if $\inf_\xi\{Z_1,Z_2\}$ is equal to or below this threshold at some time, then it  will approach $-\infty$ at some late finite time. Next, we   {impose} this  threshold as an additional lower bound on $ Z_1,Z_2 $ and prove a singularity formation with  initial data which can have arbitrarily small gradients. A key and novel technique is to combine the lower and upper   bounds of $Z_1$, $Z_2$  and the conservation of \emph{physical energy}   to control the positive terms in the equations for  $D_1Z_1$, $D_2Z_2$ so that the decay of $\inf_\xi\{Z_1,Z_2\}$ is sufficient for it to reach the  threshold that has been just proved. This then  eventually leads to loss of $C^1$ regularity in  finite time.


We have the following important comparison principle for the infimum of $Z_j(t,\cdot)$.
\begin{lem}[Strict comparison principle]\label{lem:comp}
Fix $\Tm>0$. Consider a classical solution  $$(h,u,v)\in C^1([0,\Tm)\times\bbR) $$ to the rotating shallow water    system \eqref{p:hinv} with  $C_0^1$ initial data $(h_0-1,u_0,v_0)$ so that
 $\inf_{\xi }h_0>0$.

Let  a function $m(t)\in C^1([0,\Tm))$   satisfy  the following strict differential inequality and  initial condition
\be\label{m:lem}
 \left\{
 \begin{aligned}
 &  \sup_{\xi\in\bbR} {\sqrt{h(t,\xi) } }\Big[-{1\over2}m^2(t)+\PVins-{1\over  {h(t,\xi) }}\Big]<{d\over dt}m(t) <0 \,,\\
 &   \inf_\xi \min_{j=1,2}Z_j(0,\xi) \le m(0)<0  \,.
 \end{aligned}
 \right.
\ee
 Then,   for any $t\in(0,\Tm)$,
 \be\label{Z:m}  \inf_{ \xi\in \mR }\min_{j=1,2}Z_j(t,\xi)< m(t)\,.\ee
\end{lem}
We recall that the upper and positive lower bounds of $h $ have been established, so the left hand side of \eqref{m:lem}  is always well-defined.

\begin{pf} With compactly supported initial data $(h_0-1,u_0,v_0)$,  by the bounds of $h$, which leads to the finite propagation speed of the solutions, we have that  $Z_j(t,\cdot)$ ($j=1$, $2$) is also compactly supported, so \be\label{def:ZbN}\Zf(t):=\inf_{ \xi\in\mR }\min_{j=1,2}Z_j(t,\xi) \quad\text{for } t\in [0, \Tm)\,,
\ee
is a well-defined continuous function as long as the $C^1$ solution exists.

Since $\Zf(0)\leq m(0)$ and the initial data has compact support, without loss of generality, there exists a $\xi'\in \mR$ such that $Z_1(0, \xi')\leq  m(0)<0$ and $Z_1(0, \xi')\leq Z_2(0, \xi')$. Hence one has
\begin{equation*}
\begin{aligned}
& \sqrt{h}\left(-\frac{1}{2}Z_1^2 +\bar\gamma Z_1(Z_2-Z_1)+\PVins -\frac{1}{h}\right)(0,\xi')\\
\leq & \sqrt{h}\left(-\frac{1}{2}Z_1^2 +\PVins -\frac{1}{h}\right)(0,\xi')\\
\leq & \sqrt{h}(0, \xi')\left(-\frac{1}{2}m^2(0) +\PVins -\frac{1}{h}(0,\xi')\right).
\end{aligned}
\end{equation*}
This, together with the equations \eqref{Z:n} and  \eqref{m:lem}, implies
 \[D_1Z_1(0,\xi')< {d\over dt}m(0),\]
 where the characteristic curve associated with $D_1Z_1$   emits from $( 0,\xi')$. Therefore,  there exists   a time  $T_0\in(0,T]$ so that
\be\label{ZfN:m:0}\Zf(t)< m(t)\quad\text{ for all }\;\; t\in(0,T_0).\ee
Noting that the inequality in \eqref{ZfN:m:0} is a strictly inequality in an  open interval $(0, T_0)$, we can  choose $T_0$ to be the supremum of all such time in $(0, T)$.
Then, in order to show \eqref{Z:m}, it
 suffices to prove
   $T_0=T$. We prove it by contradiction.

Suppose  $T_0<T$.  Then, we must have $m(T_0)=\Zf(T_0)$ so that by the definition of $\Zf$,
\be\label{m:Z:1}m(T_0) \le Z_j(T_0,\xi),\;\; \text{ for any }j=1,2\,\,\text{ and any }\xi\in\mR \,.\ee

Next, note that the differential inequality in  \eqref{m:lem} is valid in the closed interval $[0,T_0]$ with all its terms being continuous and with $h$ bounded from below by a positive constant. Therefore, there exists an $\ep>0$ so that
\be\label{m:lem:ep}  \sup_{\xi\in\bbR} {\sqrt{h(t,\xi) } }\Big[-{1\over2}m^2(t) +\PVins-{1\over  {h(t,\xi) }}\Big]<{d\over dt}m(t)-\ep\,,\quad\text{for all }\; t\in[0,T_0]\,.\ee

Now, for any $t_0\in(0,T_0)$, by \eqref{ZfN:m:0} and solution being compactly supported,  we assume without loss of generality that
\be\label{m:Z:2}Z_1(t_0,\xi_0)=\Zf(t_0)<m(t_0)<0\qquad\text{for some }\;\; \xi_0\in\mR \,,\ee
where the last inequality is due to the  assumption that both $m(0)$ and $m'(t)$ are negative.

Let $\Xi(t)$ be the solution of the following Cauchy problem for ODE
\[
\left\{
\begin{aligned}
&{d\over dt}\Xi(t)=-h^{\gamma+1\over2}(t,\Xi(t)),\\
& \Xi(t_0)=\xi_0.
\end{aligned}
\right.
\]
Hence $\Big\{(t,\Xi(t))\,\Big|\,t_0\le t\le T_0 \Big\}$ is the   characteristic curve associated with $D_1Z_1$.
By \eqref{m:Z:1}, \eqref{m:Z:2}, we have
\[
m(T_0)-m(t_0)  <    Z_1(T_0,\Xi(T_0))-Z_1(t_0, \xi_0)\,,
\]
which is equivalent to
\[
\int_{t_0}^{T_0}m'(t)\,dt<
 \int_{t_0}^{T_0}D_1Z_1\,dt.
 \]
Apply the equations \eqref{Z:n} for $Z_1$ to arrive at
\[\int_{t_0}^{T_0} \left\{m'(t)-
 {\sqrt h}\Big[-\gp Z_1^2+\gm Z_1Z_2+\PVin-{1\over h}\Big](t, \Xi(t))\right\}\, dt\,<0.  \]
 Since $m(t)\in C^1$, applying the intermediate value theorem  yields that for some $ \tau(t_0)\in(t_0,T_0)$,
\be\label{IMV:m}
 m'(\tau(t_0))-{\sqrt h}\Big[-\gp Z_1^2+\gm Z_1Z_2+\PVin-{1\over h}\Big](\tau(t_0),\Xi(\tau(t_0))) < 0\,.
 \ee
 Note that by the definition of $Z_1,Z_2$, the positive lower bound of $h$,  the hypothesis that $(h, u, v)$ is $C^1$,  and   the fact that the curve $(t,\Xi(t))$ is $C^1$ and contained in a compact region, we must have $Z_1(t,\Xi(t))$, $Z_2(t,\Xi(t))$,  $\sqrt{h(t,\Xi(t))}$, and $1/h(t,\Xi(t))$ to be
uniformly continuous functions of $t$ over  $[t_0,T_0]$ and the modulus of continuity is independent of the choice of $t_0$.
Therefore, for the same $\ep$ as in \eqref{m:lem:ep}, we can choose  $T_0-t_0$ to be sufficiently small, making  $\tau(t_0)-t_0$ even smaller, so that by \eqref{IMV:m},
\[
 m'(t_0)<{\sqrt h}\Big(-\gp Z_1^2+\gm Z_1Z_2+\PVin-{1\over h}\Big)(t_0,\Xi(t_0)) +{\ep}\,.
 \]
It follows from \eqref{m:Z:2} that one has
\begin{equation*}
\begin{aligned}
m'(t_0) <{\sqrt {h( t_0,\xi_0)}}\Big[-{ {1\over2} } m^2 (t_0)+ \PVin-{1\over h( t_0,\xi_0)}\Big]+{\ep}\,.
\end{aligned}
\end{equation*}
 This is a contradiction to \eqref{m:lem:ep}. The proof of the lemma is completed.
 \end{pf}

  \subsection{Existence of a threshold for formation of singularity}\label{sec31}
Now we prove Theorem \ref{thm:large:sing}, which shows that
the loss of $C^1$ regularity always takes place in finite time, provided at  some time $t$,  the minimum of $Z_j$ is below the time-independent threshold $-\sqrt{2\,\PVins}$.

\begin{proof}[Proof of Theorem \ref{thm:large:sing}]
It suffices to consider $T'=0$.

It follows from  Lemmas \ref{lem:up:h} and \ref{lem:up:low} (the estimates \eqref{max:h} and \eqref{inf:h}) that we have
\begin{equation}\label{defthetaflat}
 \sup_{\xi\in \bbR}h(t,\xi)\le  \theta^\sharp (t) \quad \text{and}\quad  \inf_{\xi\in \bbR} h(t,\xi)\ge  \Big[{1\over \sqrt{  \inf_{\xi }{  h_0}  }}+{t\over2}\Wins \Big]^{-2}=:\theta^\flat (t) \,,
\end{equation}
respectively.
Apparently $\theta^\sharp\ge\theta^\flat>0$. Now, Let $m(t)$ be the solution of the following Cauchy problem
\be\label{m:large:ODE}
{d\over dt}m(t)={\sqrt{\theta^\flat (t)} }\Big[-{1\over2}m^2(t)+\PVins-\frac{1}{2}{1\over  {\theta^\sharp (t)}}\Big]
\ee
and
\begin{equation}\label{m:large:init}
m(0)= \inf_{ \xi \in \mR }\min \big\{Z_1(0,\xi),Z_2(0,\xi)\big\}\le -\sqrt{2\PVins}.
 \end{equation}
It is easy to see that $m(t)$ is strictly decreasing and  satisfies the assumptions of Lemma \ref{lem:comp} as long as it remains finite. Therefore,
\[
\inf_{\xi}\min_{j=1,2} Z_j(t,\xi) < m(t).
\]

By monotonicity of $m(t)$, there exists a $T_1>0$ such that
\[
{1\over2}m^2(T_1)-\PVins=:a>0.
\]
Hence for any $t> T_1$, one has
\be\label{m:at:T_1}
m(t)\le m(T_1)= -\sqrt{2\PVins+2a}<-\sqrt{2\PVins}.
\ee
It follows from \eqref{m:large:ODE} that the following differential inequality holds
\begin{equation}\label{neweq57}
{d\over dt}m(t)<{\sqrt{\theta^\flat (t)} }\Big[-{1\over2}m^2(t)+\PVins \Big].
\end{equation}
Using partial fractions yields
  \[
  \frac{dm}{m-\sqrt{2\PVins}}-\frac{dm}{m+\sqrt{2\PVins}}< -\sqrt{2 \PVins \theta^\flat(t)}\,dt.
\]
Integrate this inequality from $T_1$ to $t>T_1$ with relevant signs determined by   \eqref{m:at:T_1},
\[
\ln\,\frac{m(t)-\sqrt{2\PVins}}{m(t)+\sqrt{2\PVins}}  -\ln\,   \frac{m(T_1)-\sqrt{2\PVins}}{m(T_1)+\sqrt{2\PVins}} < -\sqrt{2\PVins}\int_{T_1}^t \sqrt{  \theta^\flat(s)}\,ds.
\]
Combining with the definition of $ \theta^\flat$ in \eqref{defthetaflat} gives
\[\begin{aligned}\ln\,\frac{m(t)-\sqrt{2\PVins}}{m(t)+\sqrt{2\PVins}}\,< \,&
\frac{2\sqrt{2\PVins}}{\Wins}\,\ln\, \Big[{ 2 \over \sqrt{  \inf_{\xi }{  h_0}  }}+ \Wins\,T_1 \Big]-{\frac{2\sqrt{2\PVins}}{\Wins}}\,\ln\, \Big[{ 2 \over \sqrt{  \inf_{\xi }{  h_0}  }}+ \Wins\,t \Big]\\
&+\ln\,   \frac{m(T_1)-\sqrt{2\PVins}}{m(T_1)+\sqrt{2\PVins}}
\,,\quad¡¡\text{for } t>T_1.\end{aligned}\]
By  \eqref{m:at:T_1} again, the right side of the above expression will decrease as $t$ increases and will approach $0$ from above in finite time. This implies  $m(t)$   approaches  $-\infty$ at the same time. By the comparison principle, Lemma \ref{lem:comp},
and by Corollary \ref{cor:Tmax:open} at the end of Section \ref{sec:Rie}, the only type of singularity must satisfy \eqref{def:sing}-\eqref{def:sing2}. Hence the proof of the theorem is completed.
\end{proof}

\subsection{General initial data with small gradients}
By \eqref{sup:Z},   we always have an upper bound for $Z_j$. In order to prove the singularity formation for general initial data, it follows from Theorem \ref{thm:large:sing} that we need only to focus on the following case for the purpose of proving singularity formation,
\be\label{min:max:Z} Z_1,Z_2\in\Big( -\sqrt{2\PVins},\,\max\{\Zins ,\,\sqrt{2 \PVins
 }\}\Big].\ee
Note that the condition \eqref{min:max:Z} implies that
\[
|Z_2-Z_1|< \Gin,
\]
where  the gap $\Gin$ is defined in \eqref{def:Gin}.

 The nice thing about  \eqref{min:max:Z} is that it gives an additional bound.
In particular, considering the comparison principle in Lemma \ref{lem:comp} and especially the $-{1\over h}$ term in the differential inequality in  \eqref{m:lem}, we need a { much sharper} { upper bound} for $h$ than the  previously established one.
In fact, the \emph{a priori} assumption \eqref{min:max:Z} gives a bound on   $\pxi h$ because by definitions \eqref{def:Z},
we have
\be\label{h:G} \left|\pxi[h^{\gamma/2}(t,\xi)]\right|={\gamma\over2}\,{|Z_2-Z_1|\over2}< {\gamma\over4}\,\Gin  \,.\ee
In order to turn such estimate into an upper bound on $h$, we utilize the well-known { conservation of total physical energy}  for the rotating shallow water system.

For $C_0^1$ initial data $(h_0-1,u_0,v_0)$ and strictly positive $h$,  it is straightforward to  show that $E(t)$ defined in \eqref{E:cons} is invariant with respect to time, i.e.\[E(t)\equiv \Ein\,.\]
Immediately, by the definition and conservation of physical energy $\Ein$, for fixed $t$, we have
\be\label{f:2}
\int_{-\infty}^\infty\big(h^{\gamma/2}(t,\cdot)-1\big)^2\,d\xi\le\int_{-\infty}^\infty{\cQ(h(t,\xi))\over\zeta(\als,\gamma)}\,d\xi\le{\Ein\over \zeta(\als,\gamma)}\,,\ee
where
\begin{equation}\label{defals}
\als(t) =\displaystyle\sup_\xi(h^{\gamma/2}(t,\xi)-1)
\end{equation}
and
$\zeta$ is the function defined as follows
\begin{equation}\label{def:zeta} \zeta(\beta, \gamma) :={1 \over\gamma^2}\Big\{   (\beta+1)^{-\frac{2}{\gamma}}+(\beta+1)^{-\frac{2}{\gamma}-1}\Big\}\,.
\end{equation}
The proof of the estimate \eqref{f:2} is given in Proposition \ref{eleprop1} in Appendix \ref{secab}.
 Next, we estimate $\als(t)$  in the following lemma.

 \begin{lem} \label{lem:fss}Fix $T>0$. Consider a classical solution  $(h,u,v)\in C^1([0,T)\times\bbR) $ of the rotating shallow water    system \eqref{p:hinv} with  $C_0^1$ initial data $(h_0-1,u_0,v_0)$ so that
 $\inf_{\xi }h_0>0$. Impose the additional bound \eqref{min:max:Z} on the weighted gradients of Riemann invariants for all time $t\in[0,T)$. Then,
  \be\label{h:fss}  \als(t)< \cF_\gamma^{-1}(\Gin\,\Ein)\,,\ee
where $\als(t)$ is defined in \eqref{defals} and  $\cF_\gamma^{-1}$ is the inverse of function $\cF_\gamma$ defined by \eqref{def:cF}.
%

%

 \end{lem}

 \begin{proof}Throughout the proof, we always have uniform boundedness of $h,u,v,R_1,R_2$ and uniform strict positive lower bound of $h$ guaranteed by Lemmas \ref{lem:up:h} and \ref{lem:up:low}.

For fixed $t\in[0,T)$, it follows from Proposition \ref{easy:prop} in Appendix \ref{secab} that we have
  \begin{equation*}
  \begin{aligned}
  (\als(t))^3=&\sup_{\xi}(h^{\gamma/2}(t,\xi)-1)^3 \\
  \leq & \frac{3}{4} \,\|h^{\gamma/2}-1\|_{L^2}^2\,\big \|\pxi  \big(h^{\gamma/2}\big) \big \|_{L^\infty}\\
  < &  {3 \over 4}\, { \Ein\over \zeta(\als(t), \gamma)}\, {\gamma\over4}\,  \Gin,
  \end{aligned}
  \end{equation*}
where $\als(t)$ is the function defined in \eqref{defals} and the estimates \eqref{f:2} and \eqref{h:G} are used. Hence
  \[
  \cF_\gamma(\als)=\displaystyle{16\over3\gamma }\,{(\als)^3\, \zeta(\als,\gamma)}< \Gin\Ein.
  \]
   Therefore,  by the monotonicity   of $\cF_\gamma$, we   prove  \eqref{h:fss}.
  \end{proof}

 We are ready to state and prove the main theorem of finite time singularity formation for the solutions with arbitrarily  small initial gradients.

 \begin{thm}\label{thmenergy}
 Under the same assumptions and notations as Lemma \ref{lem:fss}, if
 \be\label{Z:main}-\sqrt{2\PVins}<\inf_\xi\big\{Z_1,Z_2\}\Big|_{t=0}< -\sqrt2\,\sqrt{\PVins-\Big[\cF_\gamma^{-1}\big(\Gin\,\Ein\big)+1\Big]^{-{2\over\gamma}}},\ee
  then $\inf_\xi\big\{Z_1,Z_2\}$ will reach $-\sqrt{2\PVins}$ at a finite time that is bounded by a continuous function of $\Gin$, $\Ein$, $\PVins$, and $\big\{Z_1,Z_2\}\Big|_{t=0}$.
 \end{thm}

\begin{pf} We prove the theorem by the contradiction argument. Suppose that the theorem is not true so that the additional gap condition \eqref{min:max:Z} is always true. Then, we can apply the estimate \eqref{h:fss} obtained in Lemma \ref{lem:fss} to have
 \[
  \sup_{\xi} h(t,x)<\hs:=\Big[\cF_\gamma^{-1}\big(\Gin\,\Ein\big)+1\Big]^{{2\over\gamma}}\,.\]
 This leads to, as long as $m(t)\in(-\sqrt{2\PVins},0)$,
$$ \sup_{\xi\in\bbR} {\sqrt{h(t,\xi) } }\Big(-{1\over2}m^2(t)+\PVins-{1\over  {h(t,\xi) }}\Big)< {\sqrt{\hs } }\Big(-{1\over2}m^2(t)+\PVins \Big)-{1\over \sqrt{\hs}} \,.$$  Then, we choose $m(t)$  to be the solution of the following initial value problem for the ordinary differential equation
 \[\left\{\begin{aligned}&{d\over dt}m(t)={\sqrt{\hs } }\Big(-{1\over2}m^2(t)+\PVins \Big)-{1\over \sqrt{\hs}}\,,\\
 &m(0)=\inf_\xi\big\{Z_1,Z_2\}\Big|_{t=0}.\end{aligned}\right.\]
 Meanwhile, by assumption \eqref{Z:main}, we have
 \[
 m(0)< -\sqrt2\,\sqrt{\PVins-{1\over \hs }}.
 \]
Then,  a straightforward calculation shows that $m(t)$ is decreasing and negative,  and reaches $-\sqrt{2\PVins}$ at a finite time. Furthermore, it is easy to see that $m(t)$ satisfies the assumptions of comparison principle, Lemma \ref{lem:comp}.
Therefore, by Lemma \ref{lem:comp}, $\inf_\xi \{Z_1, Z_2\}$ also  reaches  $-\sqrt{2\PVins}$ at a finite time. The proof of the theorem is completed.
\end{pf}

It is easy to see that Theorem \ref{mainthm1} is a direct consequence of Theorems \ref{thm:large:sing} and \ref{thmenergy}.

\section{Klein-Gordon Equation and Global Existence}\label{sec:KL}
In this section, we prove Theorem \ref{mainthm2}. For simplicity, we only consider $\gamma=2$ which is from the geophysical rotating shallow water system.

Differentiate  the third  equation  in (\ref{p:hinv}) with respect to $t$, and combine it with  the second equation in  (\ref{p:hinv}) and \eqref{PV:Lag} to obtain
\[\partial_{tt} v -\pxi\bigg(\frac{1}{2(\PVin(\xi)-\pxi v)^2}\bigg) +v =0,\]i.e.
\be\label{KL18}
\partial_{tt} v -\frac{\pxixi  v}{(\PVin(\xi)-\pxi v)^3} +v =\frac{-\PVin'(\xi)}{(\PVin(\xi)-\pxi v)^3}.
\ee
This is a typical quasilinear Klein-Gordon equation. The well-posedness for the Cauchy problem \eqref{p:hinv} and \eqref{RSW:IC} is equivalent to study the well-posedness for the  Klein-Gordon equation \eqref{KL18}.

For a general funtion $\PVin(\xi)$, the linear part of the equation \eqref{KL18}  is a Klein-Gordon operator with variable linear coefficients. There are very limited results for this type of equations because of lack of understanding for the associated linear operator. If $\PVin(\xi)$ is a constant and,
without loss of generality, we assume that $\PVin(\xi)\equiv 1$, then the equation \eqref{KL18} can be written as
\be \label{v:KG:1}
\partial_{tt} v -\frac{1}{(1-\pxi v)^3} \pxixi  v+v =0.
\ee
Note that by Taylor series
\be
\frac{1}{(1-\pxi v)^3}=1+3\,\pxi v+6( \pxi v)^{2}+\cdots
\ee
Then the system is equivalent to
\be \label{KG:v}
\partial_{tt} v - \pxixi  v+v =(3\pxi v +6(\pxi v)^2 )\pxixi v +R_4=\pxi\,\Big({3\over2}(\pxi v)^2+2(\pxi v)^3\Big)+R_4.
\ee
where $R_4$ contains quartic terms and higher order terms.

The equation \eqref{KG:v} is a typical quasilinear Klein-Gordon equation with {constant linear coefficients} and quadratic nonlinearity. It has attracted much attention in analysis since 1980's. When the spatial dimension is larger than or equal to 4, the global existence of classical solutions to  quasilinear Klein-Gordon equation with quadratic nonlinearity  was proved in  \cite{KlainermanP}. The breakthrough for   study on three dimensional Klein-Gordon equation with quadratic nonlinearity was made by Klainerman \cite{Klainerman} and Shatah \cite{Shatah} independently by using the vector field approach and normal form method, respectively. Two dimensional semilinear Klein-Gordon equation with quadratic nonlinearity was established in \cite{OzawaSL, OzawaQL} by combining the vector field approach and normal form method together. Note that the equation \eqref{KG:v} is a one dimensional quasilinear Klein-Gordon equation with quadratic nonlinearity. Since the dispersive decay rate for one dimensional Klein-Gordon equation is only $t^{-1/2}$, it is not easy to study global existence of
small solutions of Klein-Gordon equation in one dimensional setting with general nonlinearity. In \cite{Delort:1D}, Delort introduced null conditions on the structure of quadratic and cubic nonlinearities and then obtained the global existence result subject to such null conditions by performing delicate analysis with the tools of  normal form  and  vector field, and with {the hyperbolic coordinate transformation}.

We denote the quadratic and cubic nonlinearities of the Klein-Gordon equation \eqref{KG:v} as
\begin{equation}
Q(\pxi^2 v, \pxi v)=3\pxi v \pxi^2 v\quad \text{and}\quad P(\pxi^2 v, \pxi v) =6 (\pxi v)^2 \pxi^2 v.
\end{equation}
It is easy to see that $Q$ is linear with respect to $\pxi^2 v$ for fixed $\pxi v$ and $P$ is homogeneous of degree $2$ in $\pxi v$ and homogeneous of degree $1$ in $\pxi^2 v$. Let us define (here and below, primes do {\it not} indicate derivatives)
\begin{equation}
Q_1''(\pxi^2 v, \pxi v) =-i Q(-\pxi^2 v, i\pxi v)\quad \text{and}\quad P_2''(\pxi^2 v, \pxi v) =-P(-\pxi^2 v, i \pxi v),
\end{equation}
where $i=\sqrt{-1}$. If we introduce the following functions of two variables as
\begin{equation}
q_1''(\omega_0, \omega_1) =Q_1''(\omega_1^2, \omega_1) \quad \text{and}\quad p_2''(\omega_0, \omega_1) =P_2''(\omega_1^2, \omega_1),
\end{equation}
then the straightforward computations yiled
\[
q_1''(\omega_0, \omega_1) =-3\omega_1^3\quad \text{and}\quad p_2''(\omega_0, \omega_1) =-6\omega_1^4.
\]
This implies that the quantity $\Phi(y)$ defined in \cite[(1.7)-(1.9) in page 7]{Delort:1D} must be identically zero (with all other relevant $q_k'',p_k''$ being identically zero), i.e., the nonlinearity of the equation \eqref{KG:v} satisfies the null condition defined in  \cite[Definition 1.1 in page 7]{Delort:1D}. Hence it follows from \cite[Theorem 1.2]{Delort:1D} and \cite[Theorem 1.2]{Delort2} that  we prove Theorem \ref{mainthm2} for the global existence of small solutions for the rotating shallow water system. Note that the initial data in theorems  of \cite{Delort:1D,Delort2}  are in terms of $(v,\pt v)$, which is a natural choice for the Cauchy problem of the Klein-Gordon equation \eqref{KG:v}. Whereas the smallness assumptions in our Theorem \ref{mainthm2} are in terms of $(u_0,v_0)$,  they are related to $(v,\pt v)$ by the third equation in (\ref{p:hinv}).

\appendix
\section{The Lagrangian formulation for rotating shallow water system}\label{secaa}
In this appendix, we give the proof for the equivalence between the Eulerian formulation of rotating shallow water system and its Lagrangian formulation.

Suppose that $\sigma$ is defined in \eqref{ODE:xit}.
Then, for any $C^1$ scalar-valued functions $ \somef(t,x)$, we can  define $  \wt \somef(t,\xi):= \somef(t,\sigma (t,\xi))$ that satisfies the following identity,
\be \label{pt:tilde}
\pt  \wt\somef(t,\xi)=\pt  \somef(t, x)\Big|_{ x=\sigma (t, \xi)}+   u(t, x) \,\px   \somef(t,  x)\Big|_{ x=\sigma (t, \xi)}.
\ee

Suppose that $(\wt h, \wt u, \wt v)$ are defined in \eqref{def:huv}.
The mass conservation  of the (Eulerian) rotating shallow water  system as in the first equation in   \eqref{RE:Eu} becomes
\be \label{eqh2}
\pt   \wt h(t,\xi)+  \wt h(t, \xi) \px  u(t, x)\Big|_{ x=\sigma (t, \xi)}=0.
\ee
Differentiating the equation in  \eqref{ODE:xit} with respect to $\xi$ yields
\begin{equation}\label{eqpttxi}
\partial_t \pxi \sigma(t,\xi)=\pxi\,\sigma(t,\xi)\partial_x u(t, x)\Big|_{ x=\sigma (t, \xi)}.
\end{equation}
Combining the last two equations gives
\[ \pt\big[ \,\wt h(t,\xi) \,\pxi \sigma(t,\xi)\,\big] =0. \]
On the other hand, the initial condition for $\sigma$ in \eqref{ODE:xit} amounts to $\xi=\phi(\sigma(0,\xi))$. Combining it with \eqref{relabel} which we differentiate
with respect to $\xi$ yields
\be\label{relabel:cons}\qquad h(0,\sigma(0,\xi))\,\pxi \sigma (0,\xi)=1\qquad\text{i.e.}\qquad \wt h(0,\xi)\,\pxi \sigma (0,\xi)=1\,.\ee
Therefore, it follows from the last two equations that one has
\begin{equation}\label{h:pxxi}
\wt h(t,\xi)\,\pxi\sigma (t, \xi) \equiv1\qquad\text{for all }\;\;t\in[0,T).
\end{equation}
This, together with the chain rule, gives
 \[\pxi  \wt \somef(t,  \xi)={1\over \wt h(t,\xi)}\,\px   \somef(t,  x)\Big|_{ x=\sigma (t, \xi)}.\]

\begin{prop}
For $C^1$ solutions with non-vacuum initial data, the Lagrangian rotating shallow water system  \eqref{RSW:Lag} is equivalent to the original Eulerian rotating shallow water system \eqref{RE:Eu}. That is to say,
\begin{itemize}
\item[(a)] given a $C^1([0,T]\times \mR)$ solution $( h, u, v)$ of \eqref{RE:Eu} with $0<c_h\le  {h}(0,x )\le C_h$, then $(\wt h,\wt u,\wt v)$ defined by \eqref{relabel}, \eqref{ODE:xit} and \eqref{def:huv} is a solution of \eqref{RSW:Lag};
\item[(b)] given a $C^1([0,T_{\max}]\times \mR)$ solution $(\wt h,\wt u,\wt v)$  of \eqref{RSW:Lag} with $0<c_h\le  \wt{h} (0,\xi)\le C_h$, let $(h, u, v)(t,x):=(\wt h,\wt u,\wt v)(t,  \Upsilon (t, x))$ where $ \Upsilon(t, x)$ satisfies
\be\label{ODE:xi:uh}
\left\{
\begin{aligned}
&\pt   \Upsilon (t, x) =  -\wt u(t, \Upsilon (t, x))\,\wt h(t, \Upsilon (t, x)), \\
& \Upsilon(0, x)=\chi^{-1}(x)
\end{aligned}
\right.
\ee
with $\chi^{-1}$  the inverse   of   bijection $\chi$ defined as
\be\label{def:tilde:sigma}
\chi(\xi)=\int_0^\xi{1\over \wt h(0,z)}\,dz.
\ee
Then $(h, u,v)$ solves   system \eqref{RE:Eu}.
\end{itemize}
\end{prop}

\begin{proof} We  need only to prove part (b).

 By  \eqref{ODE:xi:uh},  for any $C^1$ scalar-valued functions $\wt\somef(t,\xi)$ and $\somef(t, x):=\wt \somef(t,\Upsilon (t, x))$, we have
\be\label{remark:pt}\pt  \somef(t, x)=\pt   \wt \somef(t,   \xi)\Big|_{ \xi=\Upsilon (t, x)}-\wt u(t,\xi)  \wt h(t,\xi)\,\pxi  \wt \somef(t,\xi)\Big|_{ \xi=\Upsilon (t, x)}.\ee
Hence,  the first equation in (\ref{RSW:Lag}) can be written as
\begin{equation}\label{eq121}
\pt h(t, x)=-\Big(\wt h\,\pxi(\wt u \, \wt h)\Big)\Big|_{ \xi=\Upsilon (t, x)}\,.
\end{equation}
Furthermore, it follows from \eqref{ODE:xi:uh} that one has
\begin{equation}\label{eq122}
\pt[\px\Upsilon(t, x)]=-\px\Upsilon(t, x)\pxi(\wt u\wt h)\Big|_{ \xi=\Upsilon (t, x)}.
\end{equation}
Combining \eqref{eq121} and \eqref{eq122} yields
\be\label{pt:ratio}\pt\left({\px\Upsilon(t, x)\over  h(t, x)}\right)=0,\ee
where we need $h$ to stay away from zero. On the other hand, the initial data of $\Upsilon$ in  \eqref{ODE:xi:uh} amounts to $x=\chi(\Upsilon(0,x))$.
Combine it with \eqref{def:tilde:sigma} which we differentiate with respect to $x$ to obtain
\[{\px \Upsilon(0, x)\over \wt h(0,\Upsilon(0, x))}=1,\qquad\text{i.e.,}\qquad {\px \Upsilon(0, x)\over   h(0,x)}=1.\]
This together  with \eqref{pt:ratio} implies
\[{\px\Upsilon(t, x)\over  h(t, x)}\equiv 1.\]
Thus, by the chain rule,   any $C^1$ scalar-valued function $\wt\somef(t,\xi)$ and $\somef(t, x):=\wt \somef(t,\Upsilon (t, x))$ satisfy
\be\label{pxt}  \px   \somef(t,  x)=h(t,  x)\, {\pxi  \wt \somef(t,\xi)\Big|_{ \xi=\Upsilon  (t,  x)} } \,.\ee
Substituting this into \eqref{remark:pt}  gives
\be\label{ptpt}\pt  \somef(t, x)=\pt   \wt\somef(t,   \xi)\Big|_{ \xi=\Upsilon (t, x)}- u \,\px  \somef(t,  x).\ee

Finally,  thanks to  \eqref{pxt} and \eqref{ptpt}, we have the  Lagrangian-to-Eulerian substitution rules: ``replace $\pxi$   with ${1\over   h}\,\px$ and then replace   $\pt$   with $(\pt+ u\px)$''.
 Apply them to  transform  the system  \eqref{RSW:Lag} to its Eulerian formulation \eqref{RE:Eu}.
\end{proof}

\section{Two elementary propositions}\label{secab}
In this appendix, we present two elementary propositions which are used in Section \ref{sec:shock}.
\begin{prop}\label{eleprop1}
 Given any two positive constants $\alpha$ and $\beta$ satisfying $-1<\alpha \le \beta$, then
\begin{equation}\label{Q:f2} \alpha^2\le  {\cQ((\alpha+1)^{2\over\gamma})\over{\zeta(\beta, \gamma)}} \,,
\end{equation}
where $\cQ$ and $\zeta$ are defined in \eqref{defcQ} and \eqref{def:zeta}, respectively.
\end{prop}
\begin{proof}
Define
\[
q(\alpha):= \cQ((\alpha+1)^{2\over\gamma})- { \alpha^2 \, \zeta(\beta,\gamma)}\,.
\]
By the definition of $\cQ$ and straightforward differentiation, we have
\[\begin{aligned}
q'(  \alpha)&= {2\over\gamma}(\alpha+1)^{\frac{2}{\gamma}-1}\cdot {1\over\gamma}\Big\{  \Big[{(\alpha+1)^{\frac{2}{\gamma}}}\Big]^{\gamma-2}-\Big[{(\alpha+1)^{\frac{2}{\gamma}}}\Big]^{ -2}\Big\}- {2  \alpha}\,{ \zeta(\beta, \gamma)} \\
&={2\alpha \over\gamma^2}\Big\{ (\alpha+1)^{-\frac{2}{\gamma}}+(\alpha+1)^{-\frac{2}{\gamma}-1}- (\beta+1)^{-\frac{2}{\gamma}}-(\beta+1)^{-\frac{2}{\gamma}-1}\Big\}.
 \end{aligned}\]
Since   $\alpha\in (-1, \beta]$, it is apparent from the above that
${ q'(\alpha)\alpha}\ge0$. Therefore, $q(\alpha)\ge q(0)=0$.
In other words, the estimate \eqref{Q:f2} is proved. 
\end{proof}

\begin{prop}\label{easy:prop}
Given a compactly supported function $g(\xi)\in C^1(\bbR)$, one has 
\begin{equation}\label{sharpg}
 \|g^3\|_{L^\infty}\leq \frac{3}{4}\,\|g\|_{L^2}^2\, \| g' \|_{L^\infty}\,.
\end{equation}
\end{prop}
\begin{pf}
This is a special case of the Gagliardo-Nirenberg interpolation inequality, but we prove it for completeness.

For any $\xi\in \mathbb{R}$, one has
\[
\min  \{\|g\|_{L^2(-\infty,\xi)}, \|g\|_{L^2(\xi,\infty)} \}\leq \frac{1}{2}\|g\|_{L^2(\mathbb{R})}\,.
\]
Without loss of generality, one assumes $\|g\|_{L^2(-\infty,\xi)}=\min  \{\|g\|_{L^2(-\infty,\xi)}, \|g\|_{L^2(\xi,\infty)}\}$.
Hence,
\be\label{f:a:b}\big|g^3(\xi)\big|=\Big|\int_{-\infty}^{\xi}\frac{d }{d\xi} g^3(\xi)\,d\xi\Big|\le3\,\|g\|_{L^2((-\infty,\xi) )}^2\, \| g' \|_{L^\infty((-\infty,\xi))}\leq \frac{3}{4}\,\|g\|_{L^2(\mathbb{R} )}^2\, \| g' \|_{L^\infty(\mathbb{R})}.\ee
Thus we have the desired inequality \eqref{sharpg}.
\end{pf}

\bigskip

{\bf Acknowledgement.}  Cheng would like to thank the Institute of Natural Sciences, Shanghai Jiao Tong University, for its warm invitation and kind hospitality of several visits to the institute, during which this study was carried out. Cheng would also like to thank Dr. Shengqi Yu for useful discussions. The research of Qu is supported in part by NSFC grant 11501121, Yang Fan Foundation of Shanghai on Science and Technology (no. 15YF1401100), Shanghai Key Laboratory for Contemporary Applied Mathematics at Fudan University and a startup grant from Fudan University.
The research of Xie was supported in part by  NSFC grants 11422105, and 11511140276, Shanghai Chenguang program, and the Program for
Professor of Special Appointment (Eastern Scholar) at Shanghai
Institutions of Higher Learning.

{\clrr }

\end{document}